\documentclass[a4paper,11pt,reqno]{amsart}

\usepackage{parskip}
\usepackage{amsthm,amsmath,amsfonts,amssymb}
\usepackage{cases}
\usepackage{xfrac}
\usepackage{functan}

\usepackage{color}
\usepackage{graphicx}
\usepackage{pgf, tikz-cd}
\usetikzlibrary{spy}
\usetikzlibrary{arrows}
\usepackage{mathrsfs}
\usepackage{tikzscale}
\usepackage[percent]{overpic}

\definecolor{ffffff}{rgb}{1.,1.,1.}
\definecolor{cqcqcq}{rgb}{0.75,0.75,0.75}

\usepackage[
colorlinks=true,
urlcolor=teal,
citecolor=red,
linkcolor=teal,
linktocpage,
pdfpagelabels,
bookmarksnumbered,
bookmarksopen]
{hyperref}
\usepackage[
capitalize, 
nameinlink,
noabbrev]
{cleveref}

\definecolor{grey}{rgb}{.7,.7,.7}

\renewcommand{\phi}{\varphi}

\newcommand{\Om}{\Omega}

\newcommand{\V}{\mathfrak{V}}
\newcommand{\K}{\mathfrak{K}}

\newcommand{\R}{\mathbb{R}}

\newcommand{\fr}{\partial}

\theoremstyle{plain}
\newtheorem{thm}{Theorem}[section]
\newtheorem*{thm*}{Theorem}
\newtheorem{lem}[thm]{Lemma}
\newtheorem{prop}[thm]{Proposition}
\newtheorem{cor}[thm]{Corollary}

\theoremstyle{definition}
\newtheorem{defin}[thm]{Definition}

\theoremstyle{remark}
\newtheorem{rem}[thm]{Remark}

\numberwithin{equation}{section}

\title[Isoperimetric and $p$-Cheeger sets]{Isoperimetric sets and $p$-Cheeger sets\\ are in bijection}

\author[M.~Caroccia]{Marco Caroccia}
\address[Marco Caroccia]{Politecnico di Milano,  Dipartimento di Matematica,  piazza Leonardo da Vinci 32, IT--20133 Milano}
\email{marco.caroccia@polimi.it}
\author[G.~Saracco]{Giorgio Saracco}
\address[Giorgio Saracco]{Dipartimento di Matematica, Universit\`a di Trento, via Sommarive 14, IT--38123 Povo-Trento}
\email{giorgio.saracco@unitn.it}%

\thanks{M.C.~and G.S.~are members of INdAM and have been partially supported by the INdAM--GNAMPA Project 2020 ``Problemi isoperimetrici con anisotropie'' (n.~prot.~U-UFMBAZ-2020-000798 15-04-2020) and the INdAM--GNAMPA Project 2022 ``Stime ottimali per alcuni funzionali di forma'' (codice CUP{\textunderscore}E55\-F22\-00\-02\-70\-001). G.S.~has also received funding from Universit\`a degli Studi di Trento (UNITN) under the Starting Grant Giovani Ricercatori 2021 project ``WeiCAp'', codice CUP{\textunderscore}E65\-F21\-00\-41\-60\-001.}



\makeindex
\begin{document}

\begin{abstract}
Given an open, bounded, planar set $\Omega$, we consider its $p$-Cheeger sets and its isoperimetric sets. We study the set-valued map $\mathfrak{V}:[\sfrac12,+\infty)\rightarrow\mathcal{P}((0,|\Omega|])$ associating to each $p$ the set of volumes of $p$-Cheeger sets. We show that whenever $\Omega$ satisfies some geometric structural assumptions (convex sets are encompassed), the map is injective, and continuous in terms of $\Gamma$-convergence. Moreover, when restricted to $(\sfrac 12, 1)$ such a map is univalued and is in bijection with its image. As a consequence of our analysis we derive some fine boundary regularity result.
\end{abstract}

 \hspace{-3cm}
 {
 \begin{minipage}[t]{0.6\linewidth}
 \begin{scriptsize}
 \vspace{-3cm}
This is a pre-print of an article published in {J. Geom Anal.}. The final authenticated version is available online at: \href{https://doi.org/10.1007/s12220-022-01157-x}{https://doi.org/10.1007/s12220-022-01157-x}
 \end{scriptsize}
\end{minipage} 
}

\maketitle

\section{Introduction}

Let $\Omega$ be an open, bounded subset of $\R^2$, and let $p \ge \sfrac 12$. We define the $p$-Cheeger constant of $\Omega$ as follows
\begin{equation}\label{eq:p-Cheeger}
H(p):= 
\inf \left\{\,\frac{P(F)}{|F|^p}\,:\, F\subset \Omega\,,\, |F|>0 \,\right\},
\end{equation}
where $|F|$ stands for the standard Lebesgue measure of the Borel set $F$ and $P(F)$ for its distributional perimeter, and we refer to~\cite{Mag12book} for an introduction to the theory of sets of finite perimeter. We shall denote by $E_p$ any set attaining the infimum in~\eqref{eq:p-Cheeger} and call it a \emph{$p$-Cheeger set}, refer to \cref{def:pCheeg}. 

On the one hand, the choice $p=1$ corresponds to the classic Cheeger problem widely studied in literature. For a general overview we refer the reader to the surveys~\cite{Leo15, Par11}. The problem is by now well-understood in dimension $2$, where a formula to compute $H(1)$ and a geometric characterization of minimizers is available in a wide generality. We refer to~\cite{KL06} for convex sets, to~\cite{KP11, LP16} for strips, and to~\cite{LNS17, LS20} for the most general statement. In dimension $2$, additional properties have been proved when $\Omega$ enjoys a rotational symmetry~\cite{Can22}, and we also mention that a complete characterization of the Blaschke--Santal\'o diagram for the triplet Cheeger constant, perimeter and area of $\Omega$ has been recently obtained in~\cite{Fto20}, and for more general triplets in~\cite{FMP22}. Finally, some stability results in the planar case are available in~\cite{caroccia2015note}.

The Cheeger problem can be stated in general dimension $N$, but finer characterizations are missing. We here only mention~\cite{AC09, ACV05} that establish uniqueness and convexity of the minimizer whenever $\Omega$ is convex, and~\cite{BP21} that proves rotational symmetry of minimizers whenever $\Omega$ is a set of revolution, and~\cite{CAROCCIA20221} for some finer regularity results. Explicit characterization of minimizers is available for few sets, we refer to~\cite{KLV19}, and remark that the unique minimizer is unknown even for the unit cube.

Determining the constant $H(1)$ and the minimizers is a problem that attracted a lot of attention because it is related to many others, the most known being the Cheeger's inequality, through which $H(1)$ provides a bound from below to the first eigenvalue of the Dirichlet $p$-Laplacian, and we refer to the foundational paper~\cite{Che70} (originally stated in a Riemannian framework) and to more recent improved estimates~\cite{Fto21, Par17}. The constant also appears in other spectral problems, see, e.g.,~\cite{BP18, caroccia2017cheeger, caroccia2019cheeger}. We also refer to~\cite{FPSS22} for the extension of these spectral properties in the very general context of abstract measure spaces.

On the other hand, for $p=\sfrac 12$ the functional is scaling invariant, and it reduces to determining the cases of equality in the isoperimetric inequality. Hence, in this latter case, the only minimizers are all and only the balls contained in $\Om$. For this topic, we refer to the beautiful survey~\cite{Fus15}.

Up to our knowledge, problem~\eqref{eq:p-Cheeger} has been first studied in the range $(\sfrac 12, 1]$ in~\cite{Avi97} again in relation to spectral inequalities, and some quantitative inequalities have been later proved in~\cite{FMP09a}. We also refer to the recent paper~\cite{PS17} for a more geometric point of view. In this range of exponents the perimeter plays a stronger role and moving towards $\sfrac 12$ minimizers try to be as round as possible. Nevertheless, nothing prevents one from considering exponents beyond $1$, and the basic results of~\cite{PS17} still hold.

In this paper, we are interested in the following geometric point of view. Fixed $\Omega\subset \mathbb{R}^2$, we consider the isoperimetric problem
\[
I(V) := \inf\left\{\,P(F)\,:\, F\subset \Omega,\, |F|=V \,\right\}.
\]
The characterization of sets attaining $I(V)$ has been first fully carried out in~\cite{SZ97} in the planar, convex setting, and later on extended in~\cite{LS22} to a more general class. We also mention~\cite[Cor.~2.12]{Ind20} where a first partial result, namely, the convexity of sets $E$ attaining $I(V)$, in the planar, convex case has been proved for anisotropic perimeters.

It is rather easy to see that any $p$-Cheeger set $E_p$ attains $I(|E_p|)$. If $\Omega$ is not a ball, denoting with $R$ the \emph{inradius} of $\Omega$, one easily proves that $|E_p|>\pi R^2$, whenever $p> \sfrac 12$. It is reasonable to ask if, given any volume $V\in (\pi R^2, |\Omega|)$, one can find an exponent $p>\sfrac 12$ such that there exists a $p$-Cheeger set $E_p$ with such a volume.

We recall that for $p=1$ the class of Cheeger sets is closed with respect to countable unions and intersections, and this allows to define \emph{maximal and minimal Cheeger sets}, refer to~\cite[Sect.~2]{CCN10}. An alternate definition of maximal and of minimal Cheeger sets can be given in terms of their volumes, refer to~\cite[Def.~3.5]{LS22}. There, the authors define
\begin{align}
m(\Omega) 
&:= 
\inf\{\,|E_1|\,:\, \text{$E_1$ is a $1$-Cheeger set of $\Omega$}\,\}, 
\label{eq:mOm}
\\
M(\Omega) 
&:= 
\sup\{\,|E_1|\,:\, \text{$E_1$ is a $1$-Cheeger set of $\Omega$}\,\},
\label{eq:MOm}
\end{align}
and define, resp., a minimal, resp., maximal, $1$-Cheeger set as a $1$-Cheeger set attaining $m(\Omega)$, resp., $M(\Omega)$. Such sets exist, and we refer, e.g., to~\cite[Prop.~3.6]{LS22}. In general one has 
\begin{equation}\label{eq:volumi_cheeger}
\pi R^2 \le m(\Omega) \le M(\Omega) \le |\Omega|\,
\end{equation}
being the second and third inequalities trivial, and the first one a straightforward consequence of the isoperimetric inequality and the scaling properties of the ratio $P(E)/|E|$. If $\Omega$ is a ball, all inequalities in~\eqref{eq:volumi_cheeger} are actually equalities; otherwise the first one is strict. If $\Omega$ is convex the second inequality is an equality~\cite{AC09} but there are also nonconvex sets for which one has equality,  refer for instance to~\cite[Thm.~2.3]{LS20}. Finally, there are several sets for which the last inequality is an equality, refer to~\cite{Sar21}.

Our main result is that for a quite general class of planar sets $\Omega$, refer to \cref{def:omega_nonecks}, there exists a strictly increasing, continuous function (hence, a bijection)
\[
\mathfrak{V} : (\sfrac 12, 1) \to (\pi R^2, m(\Omega)),
\]
such that a set $E$ attains $H(p)$ if and only if it attains $I(\mathfrak{V}(p))$. The analog cannot be fully established in the supercritical regime $p>1$, since we are unable to prove that the volume of a $p$-Cheeger set is uniquely detemined by the exponent $p$. In this case, we can only show that it remains defined as a multivalued map from $p>1$ to the power set of the interval $(M(\Omega), |\Omega|]$, and that such a map is injective and has a particular continuity property.

This result is in the same spirit of~\cite{LS22}, where for the same class of sets the authors consider the problem of characterizing the sets attaining
\[
F(\kappa) := \inf \left\{\,P(A)-\kappa|A|\,:\, A\subset \Omega,\, |A|\ge \pi R^2 \,\right\}.
\]
The parameter $\kappa$ geometrically represents the curvature of $\partial E\cap \Omega$, where $E$ is any set attaining $F(\kappa)$. They prove that there exists a continuous increasing map
\[
\mathfrak{K} : (\pi R^2, |\Omega|) \to (R^{-1}, \bar \kappa),
\]
such that a set $E$ attains $I(V)$ if and only if it attains $F(\mathfrak{K}(V))$, where
\[
\bar \kappa := \inf\left\{\,\kappa > R^{-1}\,:\, \text{$\Omega$ minimizes $F(\kappa)$} \,\right\}.
\]
Visually, as the volume $V$ increases, isoperimetric sets ``invade'' $\Omega$ and they are found by ``cutting'' $\Omega$'s corners with arcs of larger curvature $\mathfrak{K}(V)$. Under some additional geometric assumption on the set $\Omega$, the map $\mathfrak{K}$ is strictly increasing, refer to~\cite[Cor.~4.3~(ii)]{LS22}, and thus it defines a bijection. We also remark that the image of the map $\mathfrak{K}$ of the interval $(\pi R^2, m(\Omega))$ is the interval $(R^{-1}, H(1))$, hence the composition
\[
\mathfrak{K} \circ \mathfrak{V} : (\sfrac 12, 1) \to (R^{-1}, H(1))
\]
is an increasing, continuous function, and under some additional hypotheses a bijection. Analogously, in the supercritical regime $p>1$, one has a multivalued map into the power set of the interval $\kappa>H(1)$.

As a consequence of the fact that any set $E$ attaining either $H(p)$ or $I(V)$ also attains $F(\kappa)$ for a suitable $\kappa$, by adapting the strategy of~\cite{CAROCCIA20221}, we prove a fine regularity result on the contact set $\partial E \cap \partial \Omega$, yielding a lower bound on its Hausdorff dimension.

\subsection{Organization of the paper}

The paper is organized as follows. In \cref{sct:Not}, we set some notation, and state the main results of our paper.  In \cref{sct:tools}, we recall some known results and prove some preliminary lemmas needed in the proof of our main \cref{thm:main}, whose proof is contained in \cref{sct:pf}. Finally, in \cref{sct:pfcor} we exploit our main result to prove \cref{cor:app}, that establishes some fine regularity results on the free boundary of sets attaining either $H(p)$ or $I(V)$ or $F(\kappa)$.

\section{Notation and main results}\label{sct:Not}

Given an open, bounded set $\Omega\subset \mathbb{R}^2$ we are interested in the following three functionals of geometric flavor.

\begin{defin}[Prescribed curvature sets]\label{def:pmc}
Let $\Omega\subset \mathbb{R}^2$ be open, and bounded, and let $R$ be its inradius. Given $\kappa\ge R^{-1}$, we say that a set $E\subseteq\Omega$ is a \textit{set of prescribed curvature $\kappa$ of $\Omega$} if it attains the infimum
\begin{equation}\label{eq:pmc}
F(\kappa)
:= 
\inf \left \{\,P(A)-\kappa|A|\,:\, A\subset \Omega,\, |A| \ge \pi R^2 \,\right\}.
\end{equation}
\end{defin}

\begin{defin}[Isoperimetric sets]\label{def:iso}
Let $\Omega\subset \mathbb{R}^2$ be open, and bounded, and let $R$ be its inradius. Given $V\ge \pi R^2$, we say that a set $E\subseteq\Omega$ is an \emph{isoperimetric set of volume $V$ of $\Omega$} if it attains the infimum
\begin{equation}\label{eq:iso}
I(V):=
\inf\{\,P(F)\,:\, F\subset \Omega\,, |F|=V\,\}.
\end{equation}
\end{defin}

\begin{defin}[$p$-Cheeger sets]\label{def:pCheeg}
Let $\Omega\subset \mathbb{R}^2$ be open, and bounded. Given $p\ge \sfrac 12$, we say that a set $E\subseteq \Omega$ is a \textit{$p$-Cheeger set of $\Omega$} if it attains the infimum
\begin{equation}\label{eq:pCheeg}
H(p):= 
\inf\left\{\,\frac{P(F)}{|F|^p}\,:\, F\subset \Omega\,, |F|>0\,\right\}.
\end{equation}
\end{defin}

The sets we are interested in are those with \emph{no necks of any radius}, a concept first introduced in~\cite[Def.~1.2]{LNS17}. The precise definition is as follows.

\begin{defin}[Sets with no necks of any radius]
\label{def:omega_nonecks}
Let $\Omega\subset \mathbb{R}^2$ be a Jordan domain, that is, the open region bounded by a Jordan curve. Assume that the $2$-dimensional Lebesgue measure of $\partial \Omega$ is zero, that is, the curve delimiting $\Omega$ is not space-filling. Denoting with $R$ the inradius of $\Omega$, we say that $\Omega$ has no necks of radius $r\le R$, if it has the following property:

$\bullet$ given any two balls $B_r(x_0), B_r(x_1) \subset \Omega$, there exists a continuous curve $\gamma:[0,1] \to \mathbb{R}^2$ such that
\[
\gamma(0)=x_0,\qquad \gamma(1)=x_1, \qquad B_r(\gamma(t))\subset \Omega,\, \forall t\in [0,1].
\]

We say that $\Omega$ has no necks of any radius if the above property holds for all $r\le R$.
\end{defin}

We remark that any convex set, and more generally any star-shaped set, is a set with no necks of any radius, but there are many sets that enjoy such a property, whose boundary can be quite wild, e.g., Koch snowflakes. We can now state our main theorem that holds for the sets just introduced.

\begin{thm}\label{thm:main}
Let $\Omega\subset \mathbb{R}^2$ be a set with no necks of any radius, let $R$ be its inradius, and assume that $\Omega$ is not a ball. Being $\mathcal{P}(A)$ the power set of $A$, define the multivalued map
\[
\mathfrak{V} : [\sfrac 12, +\infty) \to \mathcal{P}((0,|\Omega|]),
\]
by setting
\begin{equation}\label{eq:def_V(p)}
\mathfrak{V}(p) := \{\,V\,:\, \text{there exists a $p$-Cheeger set $E_p$ of $\Omega$ with $|E_p|=V$} \,\}.
\end{equation}
The following hold true:
\begin{itemize}
\item[1)] a set $E$ is a $p$-Cheeger set if and only if it is an isoperimetric set of volume $|E|=V\in \mathfrak{V}(p)$;
\item[2)] $\mathfrak{V}$ is injective, and continuous in the following sense: if $p_i \to p$, and $V_{p_i}\in \mathfrak{V}(p_i)$, then, up to subsequences, $V_{p_i} \to V$, with $V\in \mathfrak{V}(p)$;
\item[3)] one has
\begin{itemize}
\item[(i)] $\mathfrak{V}(\sfrac 12) =(0, \pi R^2]$;
\item[(ii)] the restriction of $\mathfrak{V}$ to the interval $(\sfrac 12, 1)$ is univalued, and in particular, it is a strictly increasing, continuous function, inducing a bijection
\[
\mathfrak{V} : (\sfrac 12, 1) \to (\pi R^2, m(\Omega));
\]
\item[(iii)] $\mathfrak{V}(1) =[m(\Omega), M(\Omega)]$;
\item[(iv)] if $p>1$,  then $\V(p)\subseteq (M(\Omega), |\Omega|]$;
\end{itemize} 
where $m(\Omega)$ and $M(\Omega)$ have been, respectively, defined in~\eqref{eq:mOm} and~\eqref{eq:MOm}.
\end{itemize} 
\end{thm}

\begin{rem}
If $\Omega$ is a ball, one has $\pi R^2 = m(\Omega)=M(\Omega)=|\Omega|$, and the above theorem reads as follows. Any set with volume $V<\pi R^2$ is isoperimetric if and only if it is a $\sfrac 12$-Cheeger set, while the whole $\Omega$ is a $p$-Cheeger set for any $p\ge \sfrac 12$. Hence, $\mathfrak{V}$ maps $\sfrac 12$ in the interval $(0, \pi R^2]$, and any $p>\sfrac 12$ in the singleton $\pi R^2$.
\end{rem}

As a consequence of \cref{thm:main} and of~\cite{LS22} establishing a connection between isoperimetric sets and sets with prescribed curvature, which we sum up in \cref{thm:structure_LS22},  we can prove the following corollary.

 \begin{cor}\label{cor:app}
Let $\alpha\in (0,1]$ and $\Omega\subset \R^2$ be a set with no necks of any radius with $C^{1,\alpha}$ boundary and let $R$ be its inradius. Assume that $\Omega$ is not a ball and that $E\subseteq \Omega$ attains
\begin{itemize}
\item[a)] either $F(\kappa)$ for $\kappa > \sfrac 1R$,
\item[b)] or $F(\kappa)$ for $\kappa = \sfrac 1R$ and $E$ is not a ball,
\item[c)] or $I(V)$ for $V> \pi R^2$,
\item[d)] or $H(p)$ for $p > \sfrac 12 $.
\end{itemize}
Then, around any $x\in \partial E\cap \partial \Omega$ the set $E$ has boundary of class $C^{1,\alpha}$, and 
\[
\mathcal{H}^{\alpha}(\partial E\cap \partial \Omega)>0,
\]
where $\mathcal{H}^{\alpha}$ stands for the $\alpha$-dimensional Hausdorff measure.
\end{cor}
To prove this result we show how one can adapt the techniques and the strategy adopted in~\cite{CAROCCIA20221} for sets attaining $F(H(1))$ to those attaining $F(\kappa)$ for a general curvature $\kappa\ge R^{-1}$.

\begin{rem}
The equality in case c) would immediately imply that $E$ is an inball of $\Omega$, while the equality in case d) that $E$ is a ball contained in $\Omega$. While this would imply the regularity of $\partial E$, it would not be enough to infer anything on $\mathcal{H}^{\alpha}(\partial E\cap \partial \Omega)$.
\end{rem}

\begin{rem}
If $\Omega$ were a ball, then for any $\kappa\ge \sfrac 1R$  (resp., for any $p>\sfrac 12$) the only set attaining $F(\kappa)$ (resp., $H(p)$) would be $\Omega$ itself. Thus, the conclusion $\mathcal{H}^{\alpha}(\partial E\cap \partial \Omega)>0$ trivially follows. We also notice that assuming c) the statement would be emptily true, since $V$ is chosen greater than $\pi R^2$.
\end{rem}

\begin{rem}
It is reasonable to think that an analog of \cref{thm:main} and of \cref{cor:app} might hold in higher dimension, at least for convex $\Omega$, by replacing the exponent $\sfrac 12$ with the isoperimetric one $\sfrac{(N-1)}{N}$. On the one hand point~1) of \cref{thm:main} holds true with the same proof (and without any geometric assumptions on $\Omega$). The proofs of the remaining points heavily rely on the results of~\cite{LS22}, which extensively use the fact that the only curves with constant curvature are union of arcs of circle with the given curvature.
On the other hand, the regularity result of \cref{cor:app} holds for a higher dimensional $\Omega$ for sets attaining $F(\kappa)$, with $\kappa \ge \sfrac{(N-1)}{R}$, and, thanks to \cite[Sect.~4]{ACV05}, for sets attaining $I(V)$ with $V\ge m(\Omega)$ when $\Omega$ is convex and $C^{1,1}$-regular. The Hausdorff dimension to be considered though should be $N-2+\alpha$, refer to~\cite{CAROCCIA20221}.
\end{rem}

\section{Tools}\label{sct:tools}

\begin{prop}
\label{prop:properties_p_Cheeger}
Let $\Omega$ be a bounded, open set in $\R^2$. For all $p \ge \sfrac 12$, there exist $p$-Cheeger sets $E_p$, and the boundaries $\fr E_p \cap \Omega$ are union of arcs of circles of curvature
\begin{equation}
\label{eq:curvature_p_Cheeger}
\kappa_{E_p} 
= 
pH(p) |E_p|^{p-1}.
\end{equation}
Moreover, for $p>\sfrac 12$, the volume of any $p$-Cheeger set is at least $\pi R^2$, where $R$ is the \emph{inradius} of $\Omega$.
\end{prop}

\begin{proof}
In the range $\sfrac 12 \le p \le 1$, existence is proved in \cite[Thm.~3.2]{PS17}, while the relation~\eqref{eq:curvature_p_Cheeger} on the curvature in \cite[Thm.~2.2(4)]{PS17}. The same exact proofs work in the range $p>1$, since they only rely on the Direct Method and a first order expansion.

Thus, we are left to show the bound on the volume, and this easily follows by computing the ratio for balls. Indeed, for any ball $B_r$ of radius $r$, one has
\begin{equation}
\label{eq:p-Cheeger_ratio_decreasing}
\frac{P(B_r)}{|B_r|^p} 
= 
2\pi^{1-p}r^{1-2p},
\end{equation}
which is strictly decreasing in $r$, for $p>\sfrac 12$, while for $p=\sfrac 12$ it would be constant. For any volume $V< \pi R^2$, we let $r=r(V)<R$ be the radius of any ball with volume $V$. Then, for any set $F$ of volume $V$, by using the isoperimetric inequality, and~\eqref{eq:p-Cheeger_ratio_decreasing} we have
\[
\frac{P(F)}{V^p} 
\ge 
\frac{P(B_{r})}{V^p} 
> 
\frac{P(B_R)}{|B_R|^p} 
\ge 
H(p),
\]
which implies that such a set cannot be a minimizer.
\end{proof}

\begin{lem}
\label{lem:inf_Fk_decreasing}
Let $\Omega$ be a bounded, open set in $\R^2$, let $R$ be the inradius of $\Om$, and let $\kappa>0$ be fixed. Then $F(\kappa)$ defined in~\eqref{eq:pmc} is strictly decreasing as a function of $\kappa$, and it switches sign at $H(1)$.
\end{lem}

\begin{proof}
First, for any fixed $\kappa$ the infimum is finite, since it is bounded from below by $-\kappa|\Omega|$. Second, by the Direct Method, it is easy to see that the infimum is attained by some set $E_\kappa$, with positive volume. Let now $\kappa_1>\kappa_2$, and let $E_{\kappa_i}$ be sets achieving the respective minima. Then,
\begin{align*}
F(\kappa_1)
&= 
P(E_{\kappa_1}) - \kappa_1 |E_{\kappa_1}| 
\le 
P(E_{\kappa_2}) - \kappa_1 |E_{\kappa_2}|
\\
&< 
P(E_{\kappa_2}) - \kappa_2 |E_{\kappa_2}| 
= 
F(\kappa_2).
\end{align*}
We are left to show that $F(\kappa)$ has as unique zero $H(1)$. Clearly, there is at most one, since we have proved that $F(\kappa)$ is strictly decreasing. The fact that $F(H(1))=0$ is immediate by the definition of $1$-Cheeger constant and the bound on the volume of minimizers provided by \cref{prop:properties_p_Cheeger}.
\end{proof}


The next theorem recollects results from~\cite{LS22}, and it establishes a duality between the task of finding isoperimetric sets in $\Omega$ (in a given range of volumes) and finding minimizers of the prescribed curvature functional (in a given range of curvatures), under the assumption that $\Omega$ has no necks of any radius, following \cref{def:omega_nonecks}.

\begin{thm}
\label{thm:structure_LS22}
Let $\Omega\subset \R^2$ be a set with no necks of any radius, and let $R$ denote the inradius of $\Omega$. There is a continuous function $\K: V \mapsto \K(V)$ from $[\pi R^2, |\Omega|)$ to $[\sfrac 1R, +\infty)$ with the following properties:
\begin{itemize}
\item[(i)] a set $E_V$ of volume $V$ is isoperimetric if and only if it is a solution to the $\K(V)$-prescribed curvature problem, i.e., if it attains~\eqref{eq:pmc}, or equivalently if it minimizes 
\begin{equation}
\label{eq:def_Fk}
\mathcal{F}_{\K(V)}[E] 
:= 
P(E)-\K(V)|E|,
\end{equation}
among all subsets of $\Omega$ with $|E|\ge \pi R^2$;

\item[(ii)] given an isoperimetric set $E_V$ of volume $V$, the set $\fr E_V \cap \Omega$ has constant curvature, and it is equal to $\K(V)$;

\item[(iii)] given $V_2>V_1$, and two isoperimetric sets $E_{V_i}$ of these volumes, one has $\K(V_2)\ge \K(V_1)$, that is, the map is increasing. Moreover, if the strict inequality $\K(V_2)>\K(V_1)$ holds, one has $E_{V_1} \subsetneq E_{V_2}$;

\item[(iv)] $V < m(\Omega)$ if and only if $\K(V) < H(1)$,  $V > M(\Omega)$ if and only if $\K(V) >H(1)$,  while $V\in[m(\Omega), M(\Omega)]$ if and only if $\mathfrak{K}(V)=H(1)$,  where $m(\Omega)$ and $M(\Omega)$ have been, respectively, defined in~\eqref{eq:mOm} and~\eqref{eq:MOm}.
\end{itemize}
\end{thm}

For point~(i) we refer to~\cite[Thm.~2.4]{LS22}. The continuity of the map is not explicitly stated but this is shown in the proof of the same theorem. Point~(ii) follows from point~(i) and~\cite[Prop.~3.2~(i)]{LS22}. For point~(iii) we refer to~\cite[Cor.~3.12]{LS22}. Point~(iv) follows from point~(i), point~(iii) and the structure granted by~\cite[Thm.~2.3]{LS22} of minimizers of the prescribed curvature functional $\mathcal{F}_\kappa$ defined in~\eqref{eq:def_Fk}. We remark that under slightly stronger conditions on $\Omega$, the monotonicity of point~(iii) improves to a strict monotonicity, hence the function $\K$ becomes a bijection on its image, refer to~\cite[Cor.~4.3~(ii)]{LS22}.


\begin{lem}
\label{lem:Ep_isoperimetric}
Let $\Omega$ be an open, bounded set in $\R^2$, and let $p\ge \sfrac 12$. Any $p$-Cheeger set $E_p$ is an isoperimetric set in $\Omega$ for its own volume, that is,
\[
P(E_p) 
= 
\inf\{\,P(F)\,:\, F\subset \Omega\,, |F|=|E_p|\,\}.
\]
\end{lem}

\begin{proof}
%
%
%
Let $E_p$ be a $p$-Cheeger set. By definition of $p$-Cheeger constant, for any other set $F\subset \Omega$ one has
\[
H(p)
=
\frac{P(E_p)}{|E_p|^{p}}
\le 
\frac{P(F)}{|F|^{p}}.
\]
In particular, for all competitors $F$ such that $|F|=|E_p|=V$ one has
\[
\frac{P(E_p)}{V^p} 
\le 
\frac{P(F)}{V^p},
\]
that is, $P(E_p)\le P(F)$ for all $F\subset \Omega$ with its same volume. This means that $E_p$ is isoperimetric for its own volume. 
\end{proof}

\begin{lem}
\label{lem:curvatures_Ep}
Let $\Omega\subset \R^2$ be a set with no necks of any radius, and denote by $R$ the inradius of $\Omega$. Given any measurable subset $E\subset \Omega$, we denote by $\kappa_E$ the curvature of $\fr E \cap \Omega$. The following hold:
\begin{itemize}
\item[(i)] if $p\in (\sfrac 12, 1)$, for any $p$-Cheeger set $E_p$ one has
\[
\max\{\,p H(1), R^{-1}\,\} 
\le 
\kappa_{E_{p}} 
< 
H(1);
\]
\item[(ii)] if $p=1$, for any $1$-Cheeger set $E_1$ one has $\kappa_{E_{1}} = H(1)$;
\item[(iii)] if $p>1$, for any $p$-Cheeger set $E_p$ one has $\kappa_{E_{p}} \ge pH(1)$.
\end{itemize}
\end{lem}

\begin{proof}
Let $p> \sfrac 12$ be fixed and $E_p$ a $p$-Cheeger set of $\Omega$. Point~(ii) immediately follows from~\eqref{eq:curvature_p_Cheeger}, so let us focus on $p\neq 1$.

First, by the equality on the curvature~\eqref{eq:curvature_p_Cheeger} we have, for any $p$,
\begin{equation}
\label{eq:pterodattilo}
\kappa_{E_p} 
= 
pH(p) |E_p|^{p-1} 
= 
p\frac{P(E_p)}{|E_p|^p} |E_p|^{p-1} 
= 
p\frac{P(E_p)}{|E_p|} 
\ge 
pH(1)\,.
\end{equation}
Hence, point~(iii) immediately follows from the above inequality.

We are left with proving point~(i). 
By \cref{lem:Ep_isoperimetric}, $E_p$ is an isoperimetric set for its own volume, which, by  \cref{prop:properties_p_Cheeger}, is at least $\pi R^2$. Therefore by \cref{thm:structure_LS22}, we have that $E_p$ minimizes $P(F) - \kappa_{E_p}|F|$ among all subsets of $\Omega$ with $|F|\ge \pi R^2$, that is,
\begin{equation}
\label{eq:apatosauro}
F(\kappa_{E_p}) = P(E_p) - \kappa_{E_p}|E_p|
\end{equation}
holds. Moreover, \cref{thm:structure_LS22} grants the second lower bound on the curvature $\kappa_{E_p} = \K(|E_p|)\ge R^{-1}$. This, paired with~\eqref{eq:pterodattilo}, yields
\[
\max\{\,p H(1), R^{-1}\,\} \le \kappa_{E_{p}}\,.
\]
Finally, using~\eqref{eq:apatosauro}, that $p<1$, and the information on the curvature~\eqref{eq:curvature_p_Cheeger}, we have the inequality
\begin{align*}
F(\kappa_{E_p})
&
= 
P(E_p) - \kappa_{E_p}|E_p|
\\
&
>
P(E_p) - \frac{\kappa_{E_p}}{p}|E_p| 
=
P(E_p) - H(p)|E_p|^p
= 
0,
\end{align*}
where the last equality follows from the minimality of $E_p$ with respect to $H(p)$. This paired with \cref{lem:inf_Fk_decreasing} implies that $\kappa_{E_p}< H(1)$.
\end{proof}

\begin{rem}
\label{rem:remark_hyp_radii}
We remark that for point~(i) to hold, it would be enough for $\Omega$ to have no necks of radius $r$ for all $r\in (H(1)^{-1}, R]$, while for point~(iii) for all $r\in (0, H(1)^{-1})$, by using a finer version of \cref{thm:structure_LS22}, refer to~\cite[Thm.~2.3]{LS22}.
\end{rem}

The following result is crucial in proving the continuity of the map $\mathfrak{V}$, in terms of $\Gamma$-convergence, for whose definition one can refer to~\cite{Bra02book}.

\begin{thm}
\label{thm:G_convergence}
Let $\Omega\subset \R^2$ be an open, bounded set and denote by $R$ its inradius. For any fixed $p\ge \sfrac12$, and any measurable set $E\subset \Omega$ with volume $|E|\ge \pi R^2$, define the $p$-Rayleigh quotient
\[
R_{p}[E] = \frac{P(E)}{|E|^{p}}.
\]
Given $p\ge \sfrac12$ and $\{p_j\}_j$ a sequence with $p_j \ge \sfrac12$ converging to $p$, the functionals $R_{p_j}$ $\Gamma$-converge to $R_p$ in the $L^1_{\mathrm{loc}}$ topology.
\end{thm}

\begin{proof}
\emph{The $\Gamma$-$\liminf$ inequality.} Let $E$ be fixed, with $|E|\ge \pi R^2$, and let $E_j$ be \emph{any} sequence of sets converging to $E$ in $L^1_{\text{loc}}$, with $|E_j|\ge \pi R^2$. We need to show that
\[
R_p[E] \le \liminf_j R_{p_j}[E_j]\,.
\]
This is easily verified: the perimeter is lower semicontinuous with respect to such topology, thus $P(E)\le \liminf_j P(E_j)$ and the $L^1_{\mathrm{loc}}$ convergence implies,  together with $|E_j|\in [\pi R^2, |\Omega|]$, that
\begin{align*}
\left| |E_j|^{p_j}-|E|^p\right|\leq& \left| |E_j|^{p_j}-|E_j|^p\right|+\left| |E_j|^{p}-|E|^p\right|\\
\leq& |\Omega|^p \left| |E_j|^{p_j-p}-1\right|+\left| |E_j|^{p}-|E|^p\right|\rightarrow 0
\end{align*}
and the claim follows. 

\emph{The $\Gamma$-$\limsup$ inequality.} Let $E$ be fixed, with $|E|\ge \pi R^2$. We need to find a sequence of sets $E_j$ converging to $E$ in $L^1_{\mathrm{loc}}$ such that 
\[
R_p[E] \ge \limsup_j R_{p_j}[E_j]\,.
\]
Whether $P(E)=+\infty$ or not, the constant sequence $E_j=E$ clearly does the trick and satisfies the $\Gamma$-$\limsup$ inequality.
\end{proof}

\begin{rem}\label{rem_convergence_minimizers}
One of the powerful consequences of $\Gamma$-convergence is that if $E_{p_j}$ is a sequence of minimizers of the functionals $R_{p_j}$ $\Gamma$-converging to $R_p$, then any limit point (w.r.t.~the considered topology) of the sequence is a minimizer of $R_p$, refer to~\cite[Sect.~1.5]{Bra02book}. Therefore, since for any $q\ge \sfrac 12$ any minimizer of $R_{q}$ is a $q$-Cheeger set, the above theorem implies the following. Given any sequence $p_j$ with $p_j\ge \sfrac 12$ converging to $p$, consider a sequence $E_{p_j}$ of minimizers of $R_{p_j}$. By $\Gamma$-convergence, any of its limit points in $L^1_{\text{loc}}$ is a minimizer of $R_p$, that is, it is a $p$-Cheeger set. In particular, since $E_{p_j}$ is a $p_j$-Cheeger set, one has
\[
P(E_j) \le 2\sup_j h_{p_j}(\Omega) |\Omega|^{p_j},\, \qquad \forall j.
\]
Since $p_j$ is converging, it is bounded from above by some $\hat p$, and also from below by $\sfrac 12$. Thus, the RHS of the above inequality is bounded independently from $j$: choosing a ball $B_r\subset \Omega$ as a competitor to $h_{p_j}(\Omega)$, one has
\[
P(E_j) \le 4 \max\{1, r^{1-2\hat p}\} \max\{|\Omega|^{\frac 12}, |\Omega|^{\hat p}\}.
\]
By standard compactness results on the perimeter, $E_{p_j}$ converges in $L^1_{\text{loc}}$ to a set of finite perimeter $E\subset \Omega$. By $\Gamma$-convergence, $E$ is a $p$-Cheeger set.
\end{rem}


\begin{lem}
\label{prop:p_to_V}
Let $\Omega\subset \R^2$ be a set with no necks of any radius, and let $R$ be its inradius. For any fixed $p\geq \sfrac 12$ we have the following
\begin{itemize}
\item[(i)] $\V(\sfrac12)=(0,\pi R^2]$;
\item[(ii)] if $p \in (\sfrac 12, 1)$,  then $\V(p)\subseteq (\pi R^2, m(\Omega))$;
\item[(iii)]  $\V(1)=[m(\Omega),M(\Omega)]$;
\item[(iv)] if $p>1$, then $\V(p)\subseteq (M(\Omega), |\Omega|]$;
\end{itemize}
where $m(\Omega)$ and $M(\Omega)$ have been, respectively, defined in~\eqref{eq:mOm} and~\eqref{eq:MOm}.
\end{lem}

\begin{proof}

Assertion (i) follows by recalling that $\frac{P(E)}{\sqrt{|E|}}$ is scale invariant and is equivalent to the standard Euclidean isoperimetric problem, for which each admissible ball $B\subset \Omega$ is a solution.   

For each $p$ any $p$-Cheeger set is isoperimetric for its own volume thanks  to \cref{lem:Ep_isoperimetric}. Thus assertions (ii)--(iv) are a consequence of \cref{lem:curvatures_Ep} paired with assertions (i) and (iv) of \cref{thm:structure_LS22}.
\end{proof}

\begin{prop}\label{prop:unival}
Let $\Omega\subset \R^2$ be a set with no necks of any radius.  Then the restriction of $\V$ to $(\sfrac12,1)$ is univalued.
\end{prop}
\begin{proof}
Let $p\in (\sfrac12,1)$ be fixed and let $E^1_p$ and $E^2_p$ be two $p$-Cheeger sets, and denote by $V_i$ their volumes. Assume by contradiction, and up to relabeling, that $V_2>V_1$. By~\eqref{eq:curvature_p_Cheeger} we know that the product $\kappa_{E^i_p} V_i^{1-p}$ is constant. More precisely, we have
\[
\kappa_{E^i_p} V_i^{1-p} = pH(p)\,.
\]
In particular, since $p< 1$ and since we assumed $V_2>V_1$, we infer the strict inequality on the curvatures $\kappa_{E^1_p} > \kappa_{E^2_p}$. By \cref{lem:Ep_isoperimetric} the sets $E^i_p$ are isoperimetric for their own volumes, and by \cref{prop:properties_p_Cheeger} this is at least $\pi R^2$, where $R$ is the inradius of $\Omega$. Thus, we can apply \cref{thm:structure_LS22}~(ii) which immediately gives a contradiction, since it implies the opposite inequality $\kappa_{E^2_p} > \kappa_{E^1_p}$. 
\end{proof}

\begin{lem}
\label{lem:p->V_injective}
Let $\Omega\subset \R^2$ be a set with no necks of any radius. The map $\mathfrak{V}$ is injective.
\end{lem}

\begin{proof}
Let $p_1$ and $p_2$ be in $[\sfrac 12, +\infty)$, and assume that $\V(p_1)\cap  \V(p_2)\neq \emptyset$.  Let then $V \in \V(p_1)\cap  \V(p_2)$ and let $E_1$ and $E_2$ be, resp., a $p_1$-Cheeger set and a $p_2$-Cheeger set with $|E_1|=|E_2|=V$.  By \cref{lem:Ep_isoperimetric} they are isoperimetric for their own volumes, and since they have the same volume they necessarily have the same perimeters, that is, $P(E_1)=P(E_2)$. Moreover, by \cref{thm:structure_LS22}~(ii) the sets $\fr E_1 \cap \Omega$ and $\fr E_2 \cap \Omega$ also have the same curvature, that is, $\kappa_{E_1} = \kappa_{E_2}$.

Thus using~\eqref{eq:curvature_p_Cheeger}, one has the equality
\[
p_1H(p_1) V^{p_1-1} = p_2H(p_2) V^{p_2-1},
\]
and explicitly writing $H(p_i)$ as the $p_i$-Rayleigh ratio of $E_i$, one obtains
\[
p_1 \frac{P(E_1)}{V} = p_2 \frac{P(E_2)}{V}.
\]
Since $P(E_1) = P(E_2)$, it follows that $p_1=p_2$, hence the claim.
\end{proof}

\section{Proof of \cref{thm:main}}\label{sct:pf}

\begin{proof}[Proof of \cref{thm:main}]
We prove each point separately.

\textbf{Proof of Assertion 1):} fix $p\geq \sfrac12$.  If $E_p$ is a $p$-Cheeger set then it is isoperimetric for its own volume thanks to \cref{lem:Ep_isoperimetric} and the definition~\eqref{eq:def_V(p)} of $\V$.  Conversely, fixed an isoperimetric set $F$ of volume $V\in \V(p)$, there exists a $p$-Cheeger set $E'_p$ with $|E'_p|=V$.  By \cref{lem:Ep_isoperimetric} we also have that $P(E'_p)=P(F)$.  Thus, since $|E_p|=V$ then $E_p$ must be a $p$-Cheeger set as well.

\textbf{Proof of Assertion 2):} on the one hand the stated continuity of the multivalued map comes from \cref{thm:G_convergence} and \cref{rem_convergence_minimizers}. On the other hand, the injectivity comes from \cref{lem:p->V_injective}.

\textbf{Proof of Assertion 3):} Points~(i), (iii) and (iv) are consequences of \cref{prop:p_to_V}. We only need to prove point~(ii). \cref{prop:unival} and \cref{lem:p->V_injective} imply that $\V$ is univalued and injective. We are left to prove continuity, monotonicity and that the function is onto.

\textit{Continuity:} the continuity of $\V$ in the open interval $(\sfrac 12, 1)$ is a consequence of the $\Gamma$-convergence proved in \cref{thm:G_convergence} in the $L^1_{\mathrm{loc}}$ topology, coupled with the injectivity of \cref{lem:p->V_injective}. Let $\{p_j\}_j\subset (\sfrac 12, 1)$ be any sequence converging to $p \in (\sfrac 12, 1)$. By injectivity, $\V(p_j)$ (resp., $\V(p)$) is the singleton $V_{p_j}$ (resp., $V_p$). We want to prove that $V_{p_j}$ converges to $V_p$.

Consider any subsequence $V_{p_{j_k}}$. For each $k$, let $E_{p_{j_k}}$ be a $p_{j_k}$-Cheeger set, whose volume, by injectivity, is $V_{{p_{j_k}}}$. By \cref{rem_convergence_minimizers} these sets, up to taking a further subsequence $p_{j_{k_n}}$, converge in $L^1_{\text{loc}}$ to a set $E$ which is a $p$-Cheeger set, and whose volume is uniquely determined by the injectivity of $\mathfrak{V}$, that is, $|E|=V_p$. By the $L^1_{\text{loc}}$ convergence happening in the bounded $\Omega$, we have $V_{{p_{j_{k_n}}}} \to V_{p}$. Since from any subsequence of $V_{p_j}$ we can extract a converging sub-subsequence to the same number $V_p$, the whole sequence $V_{p_j}$ converges to the said number.

\textit{Monotonicity:} since $\V$ is a continuous and injective function, it is also strictly monotone. Thus, we only need to show that the function is increasing, and to this aim it is sufficient to show that
\[
\lim_{p\to 1^-}\V(p)>\lim_{p\to \frac12^+}\V(p).
\]
Let us start by taking any sequence $p_j$ converging to $\sfrac 12$. The uniform lower bound $|E_{p_j}| \ge \pi R^2$, and the $\Gamma$-convergence of \cref{thm:G_convergence} give that
\[
E_{p_j} \conv*{L^1_{\mathrm{loc}}}{} E_{\sfrac 12}, \qquad \qquad \text{with $|E_{\sfrac 12}| \ge \pi R^2$},
\]
where $E_{\sfrac 12}$ is a $\sfrac 12$-Cheeger set.  Since all $\sfrac 12$-Cheeger sets are balls contained in $\Omega$, we also have the opposite inequality $|E_{\sfrac 12}| \le \pi R^2$, thus
\[
\lim_{p\to \frac12^+}\V(p) = \pi R^2.
\]
A completely analogous argument works for the limit as $p\to1^-$, taking into account that $|E_p|\le m(\Omega)$ for all $p<1$ thanks to \cref{prop:p_to_V}, and that $m(\Omega) \le |E_1|$ for all $1$-Cheeger sets by its own definition~\eqref{eq:mOm}.

\textit{Surjectivity:} the fact that the function is onto $(\pi R^2, m(\Omega))$ now immediately follows. Indeed, since it is continuous and strictly monotonic, it is a bijection with its own image. Thus, one would only need to show that $\V((\sfrac 12, 1)) = (\pi R^2, m(\Omega))$, but this is trivial from the continuity and the evaluations of the limits
\[
\lim_{p\to \frac12^+} \V(p)= \pi R^2, \qquad \lim_{p\to 1^-} \V(p)= m(\Omega),
\]
we performed in the previous step.
\end{proof}

\subsection{On the case \texorpdfstring{$p>1$}{p ge 1}} It would be desirable to prove that $\V$ is univalued also in the supercritical regime $p>1$, and this would be enough to prove that it would be one-to-one between exponents in $(1, \bar p)$, being 
\[
\bar p := \inf\{\,p\,:\, |E_p|=|\Omega|\,\},
\]
and volumes $V\in (M(\Omega), |\Omega|]$, with the same exact proof we used for the subcritical case $p\in (\sfrac 12, 1)$. At the current stage, we are unable to exclude $\V$ to be multivalued for $p>1$. We do not have any counterexample to it being univalued but we have some hints that it might \emph{not} be in general.

In order to show it to be univalued, it would be enough to show that the function $\K(V)V^{1-p}$ appearing in equation~\eqref{eq:curvature_p_Cheeger} is strictly monotonic. We recall that the increasing function $\K(V)$ is the derivative of $I(V)$, see~\cite[Rem.~4.6]{LS22}. Assuming $I$ to be twice differentiable, we would be led to study the sign of
\begin{equation}\label{eq:compy}
\frac{I''(V)V + (1-p)I'(V)}{V^p}.
\end{equation}
Since $I$ is increasing, refer to~\cite[Thm.~2.3 and Thm.~2.4]{LS22}, and convex for $V\ge \pi R^2$, refer to~\cite[Thm.~2.5 and Rem.~4.6]{LS22}, it is clear that in the subcritical regime $p<1$ the above quantity is always nonnegative. On the contrary, in the supercritical regime $p>1$, a competition between the two terms ensues. In particular, one can cook up sets for which $I''(V)$ vanishes for intervals of volume as large as one wishes, thus making~\eqref{eq:compy} negative. Indeed, given any $\Omega$, it would be enough to glue a very thin and (as) long (as one wishes) rectangle to $\partial \Omega \setminus (\partial \Omega \cap \partial E_1)$, being $E_1$ a maximal $1$-Cheeger set.

\section{A boundary regularity result}\label{sct:pfcor}

The approach implemented here connects three geometric variational problems: $p$-Cheeger sets, sets of prescribed curvature $\kappa$, and isoperimetric sets of volume $V$. As a consequence we can use this connection to deduce general properties of solutions to these problems by analyzing the most convenient one. In this spirit, by studying the properties of $p$-Cheeger sets, we provide a boundary regularity theorem also for minimizers of $F(\kappa)$, and of $I(V)$ in given range of $\kappa$ and $V$. The proof follows the approach in \cite{CAROCCIA20221}. Since it is mostly a straightforward adaptation, we here only sketch it. 

Before proving \cref{cor:app} we need to state and prove the following lemma, which is a key tool in the strategy adopted in~\cite{CAROCCIA20221}. 

\begin{lem}
\label{lem:omega_is_a_ball}
Let $\Omega$ be a simply connected, open,  bounded set in $\R^2$.  Suppose that, for $\kappa>\sfrac1R$, there exists a ball $B\subseteq \Omega$ minimizing $F(\kappa)$. Then $\Omega$ itself is the ball $B$.
\end{lem}

\begin{proof}
Since $B\subseteq \Omega$ attains $F(\kappa)$, as defined in~\eqref{eq:pmc}, then $|B|\geq \pi R^2$. Being $B$ a ball, we necessarily have that it is an inball of $\Omega$, i.e., $|B|=\pi R^2$. Let us argue by contradiction. Up to a translation, we can assume that $\partial B\cap \Omega\neq \emptyset$. Then, the curvature of the free boundary $\partial B\cap \Omega\neq \emptyset$ must be $\kappa>\sfrac1R$, against the fact that the curvature of $\partial B$ is $\frac 1R$. Therefore, $\partial B\cap \Omega=\emptyset$ and thus, since $\Omega$ is simply connected, $B=\Omega$.
\end{proof}

\begin{proof}[Proof of \cref{cor:app}]
We split the proof in three steps. First, we prove the $C^{1,\alpha}$ regularity for sets satisfying a). Second, we prove the dimensional lower bound on the contact surface, again for sets satisfying a) or b). These two steps follow the strategy first used in~\cite{CAROCCIA20221}. Third, we exploit \cref{thm:main} to apply the first two steps for sets satisfying c), and d).

\textbf{Step one: }\textit{$C^{1,\alpha}$ regularity of the boundary}. Fix $\kappa\in [1/R, +\infty)$ and let $E$ be a set with prescribed  curvature $\kappa$. Assume the contact surface $\partial E\cap \partial\Omega$ to be nonempty, as otherwise there is nothing to prove, and fix $x$ in it. Without loss of generality, up to a translation and a rotation, we can assume that $x=0$ and that $\nu_E(x)=e_2$. Since $\Omega$ is a Jordan domain, so it is $E$ (see~\cite[Prop.~3.8~(ii)]{LS22}), thus we can describe locally their boundaries through continuous functions. In particular, let $f_E, f_{\Omega}$ be the functions satisfying
\begin{align*}
E\cap ([-r,r]\times \R)&=\{(x,y)\in \R^2 \ | \ -L \leq y\leq f_E(x)\}\\
\Omega\cap ([-r,r]\times \R)&=\{(x,y)\in \R^2 \ | \ -L \leq y\leq f_\Omega(x)\}
\end{align*}
for suitable $r,L$. Let us consider the space of $H^1$ functions on $(-r,r)$ that agree with $f_E$ on the boundary, and are bounded from above by $f_\Omega$, that is
\begin{align*}
\mathcal{C}&:=\left\{ w\in H^1(-r,r) \ | \ w-f_E \in H^1_0(-r,r),  \ w\leq f_{\Omega} \ \text{in $(-r,r)$} \right\} 
\end{align*}
and the prescribed curvature functional
\begin{align*}
G(w)&:=\int_{-r}^{r} \sqrt{1+(w')^2}\, \mathrm{d} x - \kappa\int_{-r}^r w\, \mathrm{d}x.
\end{align*}
It is immediate to see that $f_E$ minimizes such a functional among functions in $\mathcal{C}$. By classical theory of obstacle problems we have then that $f_E\in C^1$.  Now,  by arguing as in Step two and Step three of the proof of \cite[Lem.~5.1]{CAROCCIA20221} we achieve $f_E \in C^{1,\alpha}$.

\textbf{Step two: }\textit{Dimensional lower bound on the contact surface.} Fix again $\kappa\in [1/R, +\infty)$ and let $E$ be a set with prescribed  curvature $\kappa$, and, just as before, assume that the contact surface is nonempty. By Step one we have that around any $x\in \partial E\cap \partial \Omega$ the set $E$ has boundary of class $C^{1,\alpha}$. Let $\Gamma:=\partial E\cap \partial\Omega$, and assume by contradiction that $\mathcal{H}^{\alpha}(\partial E\cap \partial\Omega)=0$. Then $\partial E$ has constant curvature (equal to $\kappa$) on $\R^2\setminus \Gamma$ and $\mathcal{H}^{\alpha}(\Gamma)=0$.  We thus invoke \cite[Thm.~4.1]{CAROCCIA20221} to conclude that $\partial E$ has constant curvature (equal to $\kappa$) on $\R^2$. Hence, $E$ must be a ball. On the one hand, if $\kappa = \sfrac 1R$, this immediately contradicts case b). On the other hand, in case a) when $\kappa > \sfrac 1R$, we can use \cref{lem:omega_is_a_ball}, finding that $\Omega$ is a ball, against our starting hypothesis. Hence, in both cases a) and b) it must hold $\mathcal{H}^{\alpha}(\partial E\cap \partial \Omega)>0$.

\textbf{Step three:} Steps one and two above prove the validity of the thesis of \cref{cor:app} for sets satisfying either a) or b). In the following, we reason for sets $E$ satisfying either c) or d) such that $E\neq \Omega$, as otherwise there is nothing to prove.

The equivalence established by Assertion i) of \cref{thm:structure_LS22} allows to apply Steps one and two also to sets satisfying c). Finally, combining \cref{thm:main} with \cref{thm:structure_LS22}, we can associate to each $p$-Cheeger set $E_p$ with $p>\sfrac 12$ a curvature $\kappa_{E_p} \in[\sfrac1R,+\infty)$ such that $E_p$ attains $F(\kappa_{E_p})$ implying the validity of the boundary regularity in this case, settling point d).
\end{proof}

\begin{rem}
Points a) and b) are both needed in order to prove points c) and d). Indeed, there are sets $\Omega$ for which the following occurs. There exist an exponent $\hat p> \sfrac 12$ and a volume $\hat V> \pi R^2$, where as usual $R$ is the inradius of $\Omega$, such that $p$-Cheeger sets $E$ for $p\in (\sfrac 12, \hat p)$ and isoperimetric sets $E$ of volume $V\in (\pi R^2, \hat V)$ have as curvature of $\partial E \cap \Omega$ exactly $\sfrac 1R$, and they are not balls since their volume is greater than the one of an inball. In particular, this occurs whenever $\K^{-1}(\sfrac 1R)$ does not reduce to the lone volume $\pi R^2$. Equivalently, $\Omega$ does not have a unique inball, refer to~\cite[Thm.~2.3~(iii)]{LS22}, for instance whenever it is a rectangle (this example is explicitly treated in~\cite[Sect.~3]{PS17}).
\end{rem}

\begin{rem}
We remark that the result is sharp, in the sense that one can build sets $\Omega$ of class $C^{1,\alpha}$ such that for some $\kappa> \sfrac 1R$, one has sets $E\subset \Omega$ with prescribed curvature $\kappa$ with
\[
\mathrm{dim}_\mathcal{H}(\partial E \cap \partial \Omega) = \alpha.
\] 
This can be seen using the criterion proved in~\cite{Sar21}, arguing as in~\cite[Sect.~6]{CAROCCIA20221}.
\end{rem}

%



\bibliographystyle{plainurl}

\bibliography{p-cheeger_iso}

\begin{thebibliography}{10}

\bibitem{AC09}
F.~Alter and V.~Caselles.
\newblock Uniqueness of the {C}heeger set of a convex body.
\newblock {\em Nonlinear Anal.}, 70(1):32--44, 2009.
\newblock \href {https://doi.org/10.1016/j.na.2007.11.032}
  {\path{doi:10.1016/j.na.2007.11.032}}.

\bibitem{ACV05}
F.~Alter, V.~Caselles, and A.~Chambolle.
\newblock A characterization of convex calibrable sets in {$\mathbb{R}^N$}.
\newblock {\em Math. Ann.}, 332(2):329--366, 2005.
\newblock \href {https://doi.org/10.1007/s00208-004-0628-9}
  {\path{doi:10.1007/s00208-004-0628-9}}.

\bibitem{Avi97}
A.~Avinyo.
\newblock Isoperimetric constants and some lower bounds for the eigenvalues of
  the {$p$}-{L}aplacian.
\newblock {\em Nonlinear Anal.}, 30(1):177--180, 1997.
\newblock \href {https://doi.org/10.1016/S0362-546X(96)00229-5}
  {\path{doi:10.1016/S0362-546X(96)00229-5}}.

\bibitem{BP18}
V.~Bobkov and E.~Parini.
\newblock On the higher {C}heeger problem.
\newblock {\em J. Lond. Math. Soc. (2)}, 97(3):575--600, 2018.
\newblock \href {https://doi.org/10.1112/jlms.12119}
  {\path{doi:10.1112/jlms.12119}}.

\bibitem{BP21}
V.~Bobkov and E.~Parini.
\newblock On the {C}heeger problem for rotationally invariant domains.
\newblock {\em Manuscripta Math.}, 166(3--4):503--522, 2021.
\newblock \href {https://doi.org/10.1007/s00229-020-01260-9}
  {\path{doi:10.1007/s00229-020-01260-9}}.

\bibitem{Bra02book}
A.~Braides.
\newblock {\em {$\Gamma$}-convergence for beginners}, volume~22 of {\em Oxford
  Lecture Series in Mathematics and its Applications}.
\newblock Oxford University Press, Oxford, 2002.
\newblock \href {https://doi.org/10.1093/acprof:oso/9780198507840.001.0001}
  {\path{doi:10.1093/acprof:oso/9780198507840.001.0001}}.

\bibitem{Can22}
Antonio Ca\~{n}ete.
\newblock Cheeger sets for rotationally symmetric planar convex bodies.
\newblock {\em Results Math.}, 77(1):9, 2022.
\newblock \href {https://doi.org/10.1007/s00025-021-01539-7}
  {\path{doi:10.1007/s00025-021-01539-7}}.

\bibitem{caroccia2017cheeger}
M.~Caroccia.
\newblock Cheeger {$N$}-clusters.
\newblock {\em Calc. Var. Partial Differential Equations}, 56(2):30, 2017.
\newblock \href {https://doi.org/10.1007/s00526-017-1109-9}
  {\path{doi:10.1007/s00526-017-1109-9}}.

\bibitem{CAROCCIA20221}
M.~Caroccia and S.~Ciani.
\newblock Dimensional lower bounds for contact surfaces of {C}heeger sets.
\newblock {\em J. Math. Pures Appl. (9)}, 157:1--44, 2022.
\newblock \href {https://doi.org/10.1016/j.matpur.2021.11.010}
  {\path{doi:10.1016/j.matpur.2021.11.010}}.

\bibitem{caroccia2019cheeger}
M.~Caroccia and S.~Littig.
\newblock The {C}heeger-{$N$}-problem in terms of {BV}-functions.
\newblock {\em J. Convex Anal.}, 26(1):33--47, 2019.
\newblock URL: \url{https://www.heldermann.de/JCA/JCA26/JCA261/jca26003.htm}.

\bibitem{caroccia2015note}
M.~Caroccia and R.~Neumayer.
\newblock A note on the stability of the {C}heeger constant of {$N$}-gons.
\newblock {\em J. Convex Anal.}, 22(4):1207--1214, 2015.
\newblock URL: \url{https://www.heldermann.de/JCA/JCA22/JCA224/jca22063.htm}.

\bibitem{CCN10}
V.~Caselles, A.~Chambolle, and M.~Novaga.
\newblock Some remarks on uniqueness and regularity of {C}heeger sets.
\newblock {\em Rend. Semin. Mat. Univ. Padova}, 123:191--201, 2010.
\newblock \href {https://doi.org/10.4171/RSMUP/123-9}
  {\path{doi:10.4171/RSMUP/123-9}}.

\bibitem{Che70}
J.~Cheeger.
\newblock A lower bound for the smallest eigenvalue of the {L}aplacian.
\newblock In {\em Problems in analysis ({P}apers dedicated to {S}alomon
  {B}ochner, 1969)}, pages 195--199. Princeton Univ. Press, Princeton, N. J.,
  1970.

\bibitem{FMP09a}
A.~Figalli, F.~Maggi, and A.~Pratelli.
\newblock A note on {C}heeger sets.
\newblock {\em Proc. Amer. Math. Soc.}, 137(6):2057--2062, 2009.
\newblock \href {https://doi.org/10.1090/S0002-9939-09-09795-0}
  {\path{doi:10.1090/S0002-9939-09-09795-0}}.

\bibitem{FPSS22}
V.~Franceschi, A.~Pinamonti, G.~Saracco, and G.~Stefani.
\newblock The {C}heeger problem in abstract measure spaces.
\newblock \href {http://arxiv.org/abs/2207.00482} {\path{arXiv:2207.00482}}.

\bibitem{Fto21}
I.~Ftouhi.
\newblock On the {C}heeger inequality for convex sets.
\newblock {\em J. Math. Anal. Appl.}, 504(2):125443, 2021.
\newblock \href {https://doi.org/10.1016/j.jmaa.2021.125443}
  {\path{doi:10.1016/j.jmaa.2021.125443}}.

\bibitem{Fto20}
I.~Ftouhi.
\newblock Complete systems of inequalities relating the perimeter, the area and
  the {C}heeger constant of planar domains.
\newblock {\em Commun. Contemp. Math.}, page 2250054, In Press.
\newblock \href {https://doi.org/10.1142/S0219199722500547}
  {\path{doi:10.1142/S0219199722500547}}.

\bibitem{FMP22}
I.~Ftouhi, A.~L. Masiello, and G.~Paoli.
\newblock Sharp inequalities involving the {C}heeger constant of planar convex
  sets.
\newblock \href {http://arxiv.org/abs/2206.13158} {\path{arXiv:2206.13158}}.

\bibitem{Fus15}
N.~Fusco.
\newblock The quantitative isoperimetric inequality and related topics.
\newblock {\em Bull. Math. Sci.}, 5(3):517--607, 2015.
\newblock \href {https://doi.org/10.1007/s13373-015-0074-x}
  {\path{doi:10.1007/s13373-015-0074-x}}.

\bibitem{Ind20}
E.~Indrei.
\newblock On the equilibrium shape of a crystal.
\newblock \href {http://arxiv.org/abs/2008.02238} {\path{arXiv:2008.02238}}.

\bibitem{KL06}
B.~Kawohl and T.~Lachand-Robert.
\newblock Characterization of {C}heeger sets for convex subsets of the plane.
\newblock {\em Pacific J. Math.}, 225(1):103--118, 2006.
\newblock \href {https://doi.org/10.2140/pjm.2006.225.103}
  {\path{doi:10.2140/pjm.2006.225.103}}.

\bibitem{KLV19}
D.~Krej\v{c}i\v{r}\'{\i}k, G.~P. Leonardi, and P.~Vlachopulos.
\newblock The {C}heeger constant of curved tubes.
\newblock {\em Arch. Math. (Basel)}, 112(4):429--436, 2019.
\newblock \href {https://doi.org/10.1007/s00013-018-1282-x}
  {\path{doi:10.1007/s00013-018-1282-x}}.

\bibitem{KP11}
D.~Krej\v{c}i\v{r}\'ik and A.~Pratelli.
\newblock The {C}heeger constant of curved strips.
\newblock {\em Pacific J. Math.}, 254(2):309--333, 2011.
\newblock \href {https://doi.org/10.2140/pjm.2011.254.309}
  {\path{doi:10.2140/pjm.2011.254.309}}.

\bibitem{Leo15}
G.~P. Leonardi.
\newblock An overview on the {C}heeger problem.
\newblock In {\em New {T}rends in {S}hape {O}ptimization}, volume 166 of {\em
  Internat. Ser. Numer. Math.}, pages 117--139. Springer Int. Publ., 2015.
\newblock \href {https://doi.org/10.1007/978-3-319-17563-8_6}
  {\path{doi:10.1007/978-3-319-17563-8_6}}.

\bibitem{LNS17}
G.~P. Leonardi, R.~Neumayer, and G.~Saracco.
\newblock The {C}heeger constant of a {J}ordan domain without necks.
\newblock {\em Calc. Var. Partial Differential Equations}, 56:164, 2017.
\newblock \href {https://doi.org/10.1007/s00526-017-1263-0}
  {\path{doi:10.1007/s00526-017-1263-0}}.

\bibitem{LP16}
G.~P. Leonardi and A.~Pratelli.
\newblock On the {C}heeger sets in strips and non-convex domains.
\newblock {\em Calc. Var. Partial Differential Equations}, 55(1):15, 2016.
\newblock \href {https://doi.org/10.1007/s00526-016-0953-3}
  {\path{doi:10.1007/s00526-016-0953-3}}.

\bibitem{LS20}
G.~P. Leonardi and G.~Saracco.
\newblock Minimizers of the prescribed curvature functional in a {J}ordan
  domain with no necks.
\newblock {\em ESAIM Control Optim. Calc. Var.}, 26:76, 2020.
\newblock \href {https://doi.org/10.1051/cocv/2020030}
  {\path{doi:10.1051/cocv/2020030}}.

\bibitem{LS22}
G.~P. Leonardi and G.~Saracco.
\newblock The isoperimetric problem in $2$d domains without necks.
\newblock {\em Calc. Var. Partial Differential Equations}, 61(2):56, 2022.
\newblock \href {https://doi.org/10.1007/s00526-021-02153-9}
  {\path{doi:10.1007/s00526-021-02153-9}}.

\bibitem{Mag12book}
F.~Maggi.
\newblock {\em Sets of {F}inite {P}erimeter and {G}eometric {V}ariational
  {P}roblems}, volume 135 of {\em Cambridge Studies in Advanced Mathematics}.
\newblock Cambridge University Press, Cambridge, 2012.
\newblock \href {https://doi.org/10.1017/CBO9781139108133}
  {\path{doi:10.1017/CBO9781139108133}}.

\bibitem{Par11}
E.~Parini.
\newblock An introduction to the {C}heeger problem.
\newblock {\em Surv. Math. Appl.}, 6:9--21, 2011.
\newblock URL: \url{http://www.utgjiu.ro/math/sma/v06/v06.html}.

\bibitem{Par17}
E.~Parini.
\newblock Reverse {C}heeger inequality for planar convex sets.
\newblock {\em J. Convex Anal.}, 24(1):107--122, 2017.
\newblock URL: \url{https://www.heldermann.de/JCA/JCA24/JCA241/jca24009.htm}.

\bibitem{PS17}
A.~Pratelli and G.~Saracco.
\newblock On the generalized {C}heeger problem and an application to 2d strips.
\newblock {\em Rev. Mat. Iberoam.}, 33(1):219--237, 2017.
\newblock \href {https://doi.org/10.4171/RMI/934} {\path{doi:10.4171/RMI/934}}.

\bibitem{Sar21}
G.~Saracco.
\newblock A sufficient criterion to determine planar self-{C}heeger sets.
\newblock {\em J. Convex Anal.}, 28(3), 2021.
\newblock URL: \url{https://www.heldermann.de/JCA/JCA28/JCA283/jca28055.htm}.

\bibitem{SZ97}
E.~Stredulinsky and W.~P. Ziemer.
\newblock Area minimizing sets subject to a volume constraint in a convex set.
\newblock {\em J. Geom. Anal.}, 7(4):653--677, 1997.
\newblock \href {https://doi.org/10.1007/BF02921639}
  {\path{doi:10.1007/BF02921639}}.

\end{thebibliography}

\end{document}